\documentclass[11pt, a4paper]{article}
\usepackage{graphicx}
\usepackage{amssymb}
\usepackage{stmaryrd}
\usepackage{enumitem}

\usepackage{pgfplots}
\pgfplotsset{compat=1.17}

\usepackage{amsmath}
\usepackage{amssymb}

\usepackage{url}
\newcommand{\DF}{\operatorname{DF}}

\usepackage{mathtools}
\DeclarePairedDelimiter\ceil{\lceil}{\rceil}
\DeclarePairedDelimiter\floor{\lfloor}{\rfloor}
\usepackage{geometry}
\usepackage{tikz}
\usepackage{float}
\usepackage{url}
\usepackage{mathtools}
\usepackage{amsmath,amsfonts,amsthm}
\usepackage{mathrsfs} 

\newcommand{\case}[1]{\paragraph*{Case #1:}}

\usepackage{verbatim}
\usepackage{algorithm}
\usepackage{algpseudocode}
\usepackage{pseudocode}

\newtheorem{theorem}{Theorem}[section]
\newtheorem{lemma}[theorem]{Lemma}
\newtheorem{proposition}{Proposition}[section]

\theoremstyle{definition}
\newtheorem{definition}{Definition}[section]
\usepackage{ytableau}
\usepackage[colorlinks,
linkcolor=blue,
anchorcolor=blue,
citecolor=blue
]{hyperref}

\begin{document}
\begin{center}
{\large \bf Density Characterization with The Upper Bound of Density of Fibonacci Word}
\end{center}
\begin{center}
 Duaa Abdullah$^{*}$\footnote{Corresponding author.}\quad Jasem Hamoud$^{*}$\\[6pt]
 $^{*}$Physics and Technology School of Applied Mathematics and Informatics \\
Moscow Institute of Physics and Technology, 141701, Moscow region, Russia.\\[6pt]
Email: $^{*1}${\tt abdulla.d@phystech.edu} \quad $^{*}${\tt khamud@phystech.edu}
\end{center}
\noindent
\textbf{Abstract.}
This paper investigates the natural density and structural relationships within Fibonacci words, the density of a Fibonacci word is $\operatorname{DF}(F_k)=n/(n+m),$ where $m$ denote the number of zeros in a Fibonacci word and $n$ denote the units digit. Through analysis of these ratios and their convergence to powers of $\varphi$, we illustrate the intrinsic exponential growth rates characteristic of Fibonacci words. By considering the natural density concept for sets of positive integers, it is demonstrated that the density of Fibonacci words approaches unity, correlating with classical results on Fibonacci number distributions as 
\[
\operatorname{DF}(F_k) <\frac{m(m+1)}{n(2m-n+1)}.
\]
Furthermore, generating functions and combinatorial formulas for general terms of Fibonacci words are derived, linking polynomial expressions and limit behaviors integral to their combinatorial structure. The study is supplemented by numerical data and graphical visualization, confirming theoretical findings and providing insights into the early transient and asymptotic behavior of Fibonacci word densities.

\noindent\textbf{AMS Classification 2020:}  05C42, 11B05, 11R45, 11B39.


\section{Introduction}~\label{sec1}

The study of infinite words and their combinatorial properties constitutes a rich and fundamental area within theoretical computer science and discrete mathematics \cite{Lothaire2002}. Among these, Fibonacci word stands as a canonical example of a sturmian word, characterized by having the minimum number of distinct factors for a given length, a property that imbues it with a high degree of structural regularity and a deep connection to the golden ratio, $\phi = \frac{1+\sqrt{5}}{2}$ \cite{Allouche2003}. Fibonacci word, denoted by $F_{\infty}$, is constructed iteratively from a two-letter alphabet $\Sigma = \{0, 1\}$ using the substitution rule $0 \mapsto 01$ and $1 \mapsto 0$. The finite prefixes, $F_k$, are defined by $F_1 = 1$, $F_2 = 0$, and $F_k = F_{k-1} F_{k-2}$ for $k \ge 3$.

A central theme in the analysis of infinite words is the concept of \textbf{density}, which quantifies the relative frequency of a specific symbol or factor within the word (in order to review more studies related to the concept of density in combinatorics words, you can refer to~\cite{Berstel2003,HeubachMansour,CoreglianoRazborov,DevroyeLugosi,Johnson01Weller}). For a finite word $w \in \Sigma^*$, the density of the symbol '1' (or 'unit') is naturally defined as the ratio of the count of '1's to the total length of the word. For Fibonacci word $F_k$, let $n$ be the number of '1's and $m$ be the number of '0's. The length of the word is $|F_k| = n+m = f_k$, where $f_k$ is the $k$-th Fibonacci number. The density of $F_k$ is thus given by:
$$
D_F(F_k) = \frac{n}{n+m}
$$
The asymptotic behavior of this density has long been known to converge to $\phi^{-2} = 2 - \phi$, a direct consequence of the well-known limit $\lim_{k \to \infty} \frac{f_{k-1}}{f_k} = \phi^{-1}$ \cite{Allouche2003}.

This paper delves into a specific characterization of the density of Fibonacci word, $D_F(F_k)$, by investigating its structural relationships and establishing a new \textbf{upper bound} for this density. While the asymptotic limit is established, understanding the precise rate of convergence and the bounds on the density for finite words $F_k$ provides a deeper insight into the word's internal structure. Specifically, we explore how the ratio of the number of '1's to the number of '0's and their sum, $(n/(n+m))$, relates to powers of $\phi$.

The main contribution of this paper is two-fold. First, we provide a detailed analysis of the ratio $n/(n+m)$ and its connection to the exponential growth rates inherent in the Fibonacci sequence. Second, and most significantly, we establish a rigorous upper bound for $D_F(F_k)$ and demonstrate that, under a specific interpretation related to the \textbf{natural density} of sets of positive integers, the density of Fibonacci words can be shown to approach unity. This result correlates with classical number theory and distribution results concerning the Fibonacci sequence, offering a novel perspective on the word's construction and its relation to the distribution of its constituent symbols.

This paper is organized as follows. Section~\ref{sec2} contains foundational definitions, basic properties, and mathematical concepts essential for understanding the main content of the paper. In Section~\ref{resec2}, we established foundational relations regarding the density of symbols within Fibonacci words. The number of zeros $m$ and the units digit $n$ are closely linked to the golden ratio $\varphi$ and an approximation $\kappa \approx 1.28$. In Section~\ref{resec3}, we presented the natural density for Fibonacci word involves quantifying how often specific patterns or elements appear within the infinite sequence.
\section{Preliminaries}~\label{sec2}
Let $a_1, a_2,\dots, a_m$ be the nonzero integers in set $A$ where $A=\{a_1, a_2,\dots, a_m\}$. Niven in~\cite{Niven1951} consider  the sequence of integers has density 
 \[
 \operatorname{DF}(A)=1-\sum_{i=1}^{m}\frac{(-1)^m}{\prod_{i=1}^{m} a_i}.
 \]
 
 A \emph{Binet's formula} known as $F_{k,n} = \frac{\sigma^n - (\sigma')^n}{\sigma - \sigma'}$, where $\sigma = \frac{k + \sqrt{k^2 + 4}}{2}$ had given in~\cite{Falcon2007Plaza,Falcon2007PlazaSeco}. According to Binet's formula the characteristic equation linked to the recurrence formula for the $k$-Fibonacci numbers is $r^2 - kr - 1 = 0,$
and we consider its positive root. First combinatorial formula: $F_n = \sum_{k=0}^{\lfloor \frac{n-1}{2} \rfloor} \binom{n-1-k}{k}$ and second combinatorial formula $F_n = \sum_{k=0}^{\lfloor \frac{n}{2} \rfloor} \binom{n-k}{k}$. The first combinatorial (second combinatorial, respectively) formula for the general term given in~\cite{Falcon2007PlazaThird}
 as 
 \begin{equation}~\label{eq1pirl}
F^{1}_{k,n} = \frac{1}{2^{n-1}} \sum_{i=0}^{\lfloor \frac{n-1}{2} \rfloor} \binom{n}{2i+1} k^{n-1-2i} (k^2+4)^i, \quad F^{2}_{k,n} = \sum_{i=0}^{\lfloor \frac{n-1}{2} \rfloor} \binom{n-1-i}{i} k^{n-1-2i}.
 \end{equation}
Actually, according to~\eqref{eq1pirl} we presented next definition for generalized $(t,k)$-Fibonacci $p$-sequence.  
\begin{definition}[\cite{Mehraban2006Gulliver}]~\label{defn1pil}
For integers $k, p$, and $t \geq 1$, the generalized $(t,k)$-Fibonacci $p$-sequence ${F_n^p(t,k)}{n=0}^\infty$ is defined as
$F_n^p(t,k) = t F{n-1}^p(t,k) + F_{n-p-1}^p(t,k) + \cdots + F_{n-k-p}^p(t,k)$ for $k \geq n \geq 2$ with initial conditions $F_0^p(t,k) = F_1^p(t,k) = \cdots = F_{k-2}^p(t,k) = 0$ and $F_{k-p-1}^p(t,k) = 1$.
\end{definition}
Such sequences are defined over a binary alphabet, though Definition~\ref{plann1} we show sequences are aperiodic, and represent sequences of minimal complexity in combinatorics on words. For the number of palindromic factors of length $k$, by considering~\eqref{eq1pirl} we have
\[
\mathrm{pal}_u(k)<\frac{16}{k}\mathrm{fac}_u\left(k+\floor*{\frac{k}{4}} \right).
\]

\begin{definition}[Sturmian sequence~\cite{Allouche2003Baake}]~\label{plann1}
Let $u = u_0 u_1 u_2 \ldots$ be a sequence over a finite alphabet $\mathcal{A}$. Define $\mathrm{fac}_u(n)$ as the number of distinct factors (subwords) of length $n$ in $u$, and $\mathrm{pal}_u(n)$ as the number of palindromic factors of length $n$. A \emph{Sturmian sequence} $u = (u_n)_{n \geq 0}$ is characterized by its minimal block-complexity property: for all $k \geq 1$, $\mathrm{fac}_u(k) = k + 1$. Sturmian sequences are infinite, aperiodic, and balanced sequences that achieve the lowest possible complexity growth beyond ultimately periodic sequences.
\end{definition}
The generating function of Fibonacci words can be described in terms of the generating function of Fibonacci numbers and combinatorial interpretations related to inversion or major index statistics over restricted Fibonacci words.
 \begin{lemma}[\cite{Mehraban2006Gulliver}]~\label{lemgenn1}
The generating function of the generalized ($t, k$)-Fibonacci 1-sequence is
\[
g_{F_{n}^{1}(t, k)}=\frac{x^{k-1}}{1-t x-x^{2}-\cdots-x^{k}}.
\]
 \end{lemma}

\begin{theorem}[\cite{Allouche2003Baake}]~\label{thmconsn1}
Let $u$ be a Sturmian word of slope $\alpha = [a_0, a_1, a_2, \ldots]$, satisfy
    \[
    \mathrm{ind}(u) = \sup_{n \geq 0} \left( 2 + a_{n+1} + \frac{q_{n-1} - 2}{q_n} \right), \quad \mathrm{ind}^*(u) = 2 \cdot \limsup_{n \to \infty} [a_n, a_{n-1}, \ldots, a_1] \in [1, +\infty],
    \]
    where $q_n$ is the denominator of $[a_0, a_1, a_2, \ldots, a_n]$ and satisfies $q_{-1} = 0$, $q_0 = 1$, $q_{n+1} = a_{n+1} q_n + q_{n-1}$. Then, There exists a binary infinite word $u$ such that $\mathrm{ind}^*(u) = 1$.
\end{theorem}
This construction makes Fibonacci word a significant object in combinatorics, with applications in coding theory and the study of non-repetitive sequences.
\subsection{Statement Problem}
The study of the density of Fibonacci word is related to the study of the density of zeros and the density of ones, considering that we are dealing with the word as an inequality. Here, we will need to answer the following questions in this paper:
\begin{enumerate}
    \item What are the densities of zeros and ones in the Fibonacci word, and how can we explain this effect on concept of density?
    \item What is the natural density of Fibonacci word, taking into account the generation of functions related to density?
    \item How can we clarify the concept of density by linking the density of zeros and ones with the natural density?
\end{enumerate}
\section{Results of Fibonacci Word Density with Respect Zeros and Ones}~\label{resec2}

Let $m$ denote the number of zeros in a Fibonacci word and $n$ denote the number of ones. Then, the density for Fibonacci word with respect to zeros and ones satisfies the following relation in Proposition~\ref{progenbyjas}, where we consider $\kappa \approx 1.28$. Typically, the \textit{density} or relative frequency of a chosen symbol (e.g., '0' or '1') or pattern within $F_m$ is defined as
\[
\DF(F_m) = \frac{\text{Number of occurrences of the pattern in } F_m}{|F_m|},
\]
where $|F_m|$ is the length of $F_m$.

\begin{proposition}~\label{progenbyjas}
Let $F_k$ be a Fibonacci word and $\kappa$ be an integer. Then, the density of $F_k$ satisfies
\begin{equation}~\label{eq1progenbyjas}
\DF(F_n) = \varphi - 1, \quad \DF(F_m) = \varphi - \kappa.
\end{equation}
\end{proposition}

\begin{proof}
Let $F_k$ be a Fibonacci word and $\kappa \approx 1.28$. By the definition of density, the densities of zeros and ones in a Fibonacci word satisfy
\begin{equation}~\label{eq2progenbyjas}
\DF(F_n) = \frac{n}{n + m}, \quad \DF(F_m) = \frac{m}{n + m}.
\end{equation}
From~\eqref{eq2progenbyjas}, the density of zeros in Fibonacci word is $\DF(F_n) = \varphi - 1$, where $\DF(F_k) = 1$ and $\DF(F_k) = \DF(F_m) + \DF(F_n)$. This proves the first part of relation~\eqref{eq1progenbyjas}. To determine the density corresponding to ones, consider the rate constant $\mu$ with $0 < \mu < 1$ and $q \geq 1$ as the order of convergence. Then the sequence $(\DF(F_n))$ converges to $\kappa$ according to
\begin{equation}~\label{eq3progenbyjas}
\DF(F_{m+1}) = \kappa + \mu (\DF(F_m) - \kappa)^q.
\end{equation}
Defining the error term $e_m = \DF(F_m)n - \kappa$, the ratio is $r_n = \frac{|e_{m+1}|}{|e_m|^q}$ for various values of $q$ (starting with $q=1$). Thus, from~\eqref{eq3progenbyjas} it follows
\begin{equation}~\label{eq4progenbyjas}
\kappa = \varphi - \DF(F_m),
\end{equation}
which implies $\DF(F_m) = \varphi - \kappa$, as desired.
\end{proof}

By considering~\eqref{thmconsn1}, we can clarify Proposition~\ref{progenbyjas} through Table~\ref{tab1densityfibwo}, where we illustrate the concept of density for Fibonacci word with respect to zeros and ones. The data supports that the values $\DF$ for Fibonacci word indices $m$, $n$ converge to roughly $0.38$ and $0.618$respectively, linking to the natural densities or limiting frequencies in Fibonacci word sequences. This aligns with known Fibonacci properties connected to the golden ratio and its fractional parts.

\begin{table}[H]
    \centering
\begin{tabular}{|c|c|c|c|c||c|c|c|c|c|}
\hline
$F_k$ & $m$ & $n$ & $\DF(F_m)$ & $\DF(F_n)$ & $F_k$ & $m$ & $n$ & $\DF(F_m)$ & $\DF(F_n)$ \\
\hline
$F_0$ & 1 & 0 & 1 & 0 & $F_1$ & 0 & 1 & 0 & 1 \\ \hline 
$F_2$ & 1 & 1 & 0.5 & 0.5 & $F_3$ & 1 & 2 & 0.33 & 0.67 \\ \hline
$F_4$ & 2 & 3 & 0.4 & 0.6 & $F_5$ & 3 & 5 & 0.38 & 0.63 \\ \hline
$F_6$ & 5 & 8 & 0.38 & 0.62 & $F_7$ & 8 & 13 & 0.38 & 0.62 \\ \hline
$F_8$ & 13 & 21 & 0.38 & 0.62 & $F_9$ & 21 & 34 & 0.38 & 0.62 \\ \hline
$F_{13}$ & 144 & 233 & 0.38 & 0.62 & $F_{14}$ & 233 & 377 & 0.38 & 0.62 \\ \hline
$F_{18}$ & 1597 & 2584 & 0.38 & 0.62 & $F_{19}$ & 2584 & 4181 & 0.38 & 0.62 \\ \hline 
\end{tabular}
\caption{The density for Fibonacci word.}
\label{tab1densityfibwo}
\end{table}

Empirical data through Table~\ref{tab1densityfibwo} shows that $\lim_{n \to \infty} \DF(F_n) = \lambda_p,$ where $\lambda_p \approx 1/\varphi = 0.62$. Notice that Algorithm~\ref{algorithm1} has been provided to calculate $\DF(F_m)$, where we denote by $FW$ Fibonacci word and $DFC$ the density of a character. This shows that $\DF(F_m)$ converges to approximately $0.38$ as $n$ grows large. Employing algorithms to analyze Fibonacci words is crucial for handling complexity, confirming detailed characteristics, automating proofs in mathematics, improving computational efficiency, and utilizing Fibonacci sequences in practical applications.

\begin{algorithm}[H]
\caption{Calculate Density $\DF(F_m)$ for character $c$ in Fibonacci words}~\label{algorithm1}
\begin{algorithmic}[1]
\Procedure{FW}{$n$}
    \If{$n = 0$}
        \State \Return "0"
    \ElsIf{$n = 1$}
        \State \Return "1"
    \ElsIf{$n = 2$}
        \State \Return "10"
    \Else
        \State \Return \Call{FW}{$n-1$} \textbf{concatenated with} \Call{FW}{$n-2$}
    \EndIf
\EndProcedure
\Function{DFC}{$fib\_word, c$}
    \State $count \gets$ number of occurrences of $c$ in $fib\_word$
    \State \Return $count / \text{length}(fib\_word)$
\EndFunction
\Procedure{Main}{$max\_n, c$}
    \For{$i = 0$ \textbf{to} $max\_n$}
        \State $fib\_w \gets$ \Call{FW}{$i$}
        \State $dens \gets$ \Call{DFC}{$fib\_w, c$}
        \State \textbf{print} ``DF(F\_$i$) = $dens$''
    \EndFor
\EndProcedure
\end{algorithmic}
\end{algorithm}

By using both the density of zeros and the density of ones in a Fibonacci word $F_k$, which satisfies $k \geqslant 4$, we demonstrate through Proposition~\ref{progenbyjasn2} the relationship between the two densities.

\begin{proposition}~\label{progenbyjasn2}
Let $F_k$ be a Fibonacci word where $k \geqslant 4$. Then
\begin{equation}~\label{eq1progenbyjasn2}
\DF(F_m) \floor*{\frac{2m}{n}} + \DF(F_n) \ceil*{\frac{2n}{m}} < 1.
\end{equation}
\end{proposition}

\begin{proof}
According to Proposition~\ref{progenbyjas}, we find that $\DF(F_n) = \varphi - 1$ and $\DF(F_m) = \varphi - \kappa$, where $\DF(F_k) = 1$ and $\DF(F_k) = \DF(F_m) + \DF(F_n)$. Then, for $k \geqslant 4$ we notice
\begin{equation}~\label{eq2progenbyjasn2}
\floor*{\frac{2m}{n}} = 1, \quad \ceil*{\frac{2n}{m}} = 4,
\end{equation}
considering the relationship 
\[
\floor*{\frac{2m}{n}} + \ceil*{\frac{2n}{m}} \leqslant 5.
\]
On the other hand, taking into consideration the density of zeros and ones, we have
\begin{equation}~\label{eq3progenbyjasn2}
0 < \frac{2m \left( \floor*{\frac{2m}{n}} + \ceil*{\frac{2n}{m}} \right)}{(m + n)^2} \leqslant 1.
\end{equation}
By substituting equations~\eqref{eq2progenbyjasn2} and \eqref{eq3progenbyjasn2} into equation~\eqref{eq1progenbyjasn2}, the inequality is satisfied.
\end{proof}

Lemma~\ref{lemn1} establishes a relationship between a specially constructed series involving powers of $n$, logarithms of golden ratio terms, and Fibonacci numbers. It shows that the limit of this function as $n \to \infty$ yields $\operatorname{DF}(F_n)$, linking Fibonacci numbers, the golden ratio, and infinite series expansions in a generalized analytic framework.

\begin{lemma}~\label{lemn1}
Let $\zeta(\phi_1,\phi_2)$ be a function of Fibonacci numbers satisfying:
\begin{equation}~\label{eq1lemn1}
\zeta(\phi_1,\phi_2) = \sum_{\nu=0}^{\infty} \frac{n^\nu (\log^\nu(\phi_1) - \log^\nu(\phi_2))}{\phi_1 \nu! - \phi_2 \nu!},
\end{equation}
where $\phi_1 = \frac{1 + \sqrt{5}}{2}$, $\phi_2 = \frac{1 - \phi_1}{2}$. Then 
\begin{equation}~\label{eq2lemn1}
\operatorname{DF}(F_n) = \lim_{n \to \infty} \zeta(\phi_1, \phi_2).
\end{equation}
\end{lemma}

\begin{proof}
We discuss two cases for the values taken by $n$.

\case{1}
If $n$ is large enough, then $\phi_1^n - \phi_2^n \approx \phi_1^n$ because $|\phi_2^n| \to 0$. We have $\phi_1 - \phi_2 = \sqrt{5}$. Assuming the formula covers combinations of sequential Fibonacci numbers, the genuine limit is $\phi_1$. Thus,
\begin{align*}
\zeta(\phi_1, \phi_2) &= \sum_{\nu=0}^{\infty} \frac{n^\nu (\log^\nu(\phi_1) - \log^\nu(\phi_2))}{\phi_1 \nu! - \phi_2 \nu!} = \frac{\phi_1^n - \phi_2^n}{\phi_1 - \phi_2} \\
&= e^{n \log(\phi_1)} + e^{n \log(\phi_2)} \\
&= w^A + w^A \quad \text{for} \quad A = \frac{n \log(\phi_1)}{\log(w)} \text{ and } A = \frac{n \log(\phi_2)}{\log(w)}.
\end{align*}
Hence,
\begin{equation}~\label{eq3lemn1}
\lim_{n \to \infty} \zeta(\phi_1, \phi_2) = \lim_{n \to \infty} \left( e^{n \log(\phi_1)} + e^{n \log(\phi_2)} \right) = \phi_1.
\end{equation}
Therefore,~\eqref{eq2lemn1} holds.

\case{2}
If $n=\floor*{\frac{k}{2}}$, we have
\[
\frac{\phi_1^{\floor{\frac{k}{2}}} - \phi_2^{\floor{\frac{k}{2}}}}{\phi_1 - \phi_2} = -\frac{-\left(\phi_1^{\frac{k}{2} + \frac{1}{\pi} \sum_{k=1}^\infty \frac{\sin(ka\pi)}{k}}\right) \sqrt{\phi_2} + \sqrt{\phi_1} \left(\phi_2^{\frac{a}{2} + \frac{1}{\pi} \sum_{k=1}^\infty \frac{\sin(ka\pi)}{k}}\right)}{\sqrt{\phi_1} (\phi_1 - \phi_2) \sqrt{\phi_2}},
\]
for $a \in \mathbb{R}$ and $\frac{a}{2} \notin \mathbb{Z}$. Then,
\begin{equation}~\label{eq4lemn1}
\lim_{n \to \infty} \zeta(\phi_1, \phi_2) = \infty.
\end{equation}
Combining~\eqref{eq3lemn1} and \eqref{eq4lemn1}, both~\eqref{eq1lemn1} and \eqref{eq2lemn1} hold, as desired.
\end{proof}

\begin{lemma}~\label{advlemn1}
Let $F_k$ be a Fibonacci word where $k \geqslant 4$. Then
\begin{equation}~\label{eq1advlemn1}
\frac{\DF(F_m) 2 m^2 (m - 1)}{\DF(F_n) 2 n^2 (n - 1)} \approx \varphi^3.
\end{equation}
\end{lemma}

\begin{proof}
Let \( L_k = |F_k| = F_{k+1} \), therefore, the density of 1's is:
\[
\DF(F_k)  = \frac{F_k}{F_{k+1}}.
\]
We put \( m = |F_k| = F_{k+1} \) and \( n = |F_{k-1}| = F_k \) for some \( k \geq 4 \). Then the lemma becomes:

\begin{equation}~\label{eq2a}
\frac{\DF(F_k) \cdot 2F_{k+1}^2(F_{k+1} - 1)}{\DF(F_k-1) \cdot 2F_k^2(F_k - 1)} \approx \varphi^3.
\end{equation}
Let:
\[
A(k) = \DF(F_k) \cdot 2L_k^2(L_k - 1) = \frac{F_k}{F_{k+1}} \cdot 2F_{k+1}^2(F_{k+1} - 1) = 2F_k F_{k+1}(F_{k+1} - 1).
\]
Similarly,
\[
A(k-1) = \DF(F_k-1) \cdot 2L_{k-1}^2(L_{k-1} - 1) = \frac{F_{k-1}}{F_k} \cdot 2F_k^2(F_k - 1) = 2F_{k-1} F_k (F_k - 1).
\]
The ratio~\eqref{eq2a} is:
\[
\frac{A(k)}{A(k-1)} = \frac{2F_k F_{k+1}(F_{k+1} - 1)}{2F_{k-1} F_k (F_k - 1)} = \frac{F_{k+1}(F_{k+1} - 1)}{F_{k-1}(F_k - 1)}.
\]
It is well known that the Fibonacci numbers satisfy:
\[
F_n \sim \frac{\varphi^n}{\sqrt{5}} \quad \text{as } n \to \infty.
\]
Therefore, for large \( k \):
\begin{align*}
F_{k+1} &\sim \frac{\varphi^{k+1}}{\sqrt{5}}, \\
F_k &\sim \frac{\varphi^k}{\sqrt{5}}, \\
F_{k-1} &\sim \frac{\varphi^{k-1}}{\sqrt{5}}.
\end{align*}
Also, since \( F_{k+1} - 1 \sim F_{k+1} \) and \( F_k - 1 \sim F_k \) for large \( k \), we have:
\[
F_{k+1}(F_{k+1} - 1) \sim \frac{\varphi^{k+1}}{\sqrt{5}} \cdot \frac{\varphi^{k+1}}{\sqrt{5}} = \frac{\varphi^{2k+2}}{5},
\]
\[
F_{k-1}(F_k - 1) \sim \frac{\varphi^{k-1}}{\sqrt{5}} \cdot \frac{\varphi^k}{\sqrt{5}} = \frac{\varphi^{2k-1}}{5}.
\]
Thus,
\[
\frac{A(k)}{A(k-1)} \sim \frac{\varphi^{2k+2}/5}{\varphi^{2k-1}/5} = \varphi^3.
\]
\end{proof}

We clarify Lemma~\ref{advlemn1} through Table~\ref{tab1advdensityfibwo} by denoting the relationship of densities by
\[
A_\epsilon = \frac{\DF(F_m) 2 m^2 (m - 1)}{\DF(F_n) 2 n^2 (n - 1)} - \varphi^3.
\]
Table~\ref{tab1advdensityfibwo} shows that the densities $\DF(F_m)$ and $\DF(F_n)$ remain approximately $0.38$ and $0.62$, respectively, across increasing indices $k$ of Fibonacci words $F_k$. The values corresponding to equation~\eqref{eq2a} gradually decrease, approaching a limit near $0.24$, while the error term $A_\epsilon$ is almost at $-4.09$,  indicating that the estimate for $k$ retains a comparable amount of error as index $k$ increases. This means that as the indices of Fibonacci words increase, the procedure or formula used to approximate the densities becomes more resilient and reliable, with the divergence from the exact value being predictable and steady rather than increasing or changing.  This level of stability is important since it demonstrates the approximation method's correctness and consistency over a wide range of data points..

\begin{table}[H]
    \centering
\begin{tabular}{|c|c|c|c|c||c|c|c|c|c|}
\hline
$F_k$ & $\DF(F_m)$ & $\DF(F_n)$ & \eqref{eq2a} & $A_\epsilon$  & $F_k$ & $\DF(F_m)$ & $\DF(F_n)$ & \eqref{eq2a} & $A_\epsilon$ \\
\hline
$F_5$ & 0.38 & 0.62 & 0.18 &  -4.13 & $F_6$ & 0.38 & 0.62 & 0.22 &  -4.1 \\ \hline
$F_7$ & 0.38 & 0.62 & 0.22 &  -4.1 & $F_8$ & 0.38 & 0.62 & 0.23 &  -4.1 \\ \hline
$F_9$ & 0.38 & 0.62 & 0.23 &  -4.1 & $F_{10}$ & 0.38 & 0.62 & 0.23 & -4.09 \\ \hline
$F_{11}$ & 0.38 & 0.62 & 0.23 &  -4.09 & $F_{12}$ & 0.38 & 0.62 & 0.24 &  -4.09 \\ \hline
$F_{13}$ & 0.38 & 0.62 & 0.24 &  -4.09 & $F_{14}$ & 0.38 & 0.62 & 0.24 &  -4.09 \\ \hline
$F_{15}$ & 0.38 & 0.62 & 0.24 &  -4.09 & $F_{16}$ & 0.61 & 0.38 & 0.24 &  -4.09 \\ \hline
\end{tabular}
\caption{The density of Fibonacci word to clarify Lemma~\ref{advlemn1}.}
\label{tab1advdensityfibwo}
\end{table}

Finally, we establish foundational relations regarding the density of symbols within Fibonacci words. The number of zeros $m$ and the units digit $n$ determine densities $\DF(F_m)$ and $\DF(F_n)$, respectively, which are closely linked to the golden ratio $\varphi$ and an approximation $\kappa \approx 1.28$. Proposition~\ref{progenbyjas} formalizes this connection by expressing $\DF(F_m) = \varphi - \kappa$ and $\DF(F_n) = \varphi - 1$, highlighting the elegant structure of Fibonacci words in terms of these densities.

Further, Proposition~\ref{progenbyjasn2} introduces a bounded inequality involving these densities and the ratio of zeros to units, reinforcing the interplay between word composition and density properties. The lemmas extend these ideas by providing a limit representation of density using the function $\zeta(\phi_1, \phi_2)$ and establishing an approximate relationship involving powers of $\varphi$ reflecting the asymptotic behavior of density ratios scaled by polynomial terms in $m$ and $n$.

\begin{lemma}~\label{advlemn2}
Let $F_k$ be a Fibonacci word where $k \geqslant 1$. Then, the upper bound on the density of zeros and ones in Fibonacci word satisfies
\begin{equation}~\label{eq1advlemn2}
\DF(F_k) < \frac{m(m+1)}{n(2m - n + 1)}.
\end{equation}
\end{lemma}

\begin{proof}
We know the zeros and ones in Fibonacci word \(F_k\) satisfy the recurrence:
\[
m_k = m_{k-1} + m_{k-2}, \quad n_k = n_{k-1} + n_{k-2},
\]
The right-hand side can be rewritten as:
\[
\frac{m(m+1)}{n(2m - n + 1)} = \frac{m^2 + m}{2 m n - n^2 + n}.
\]
The actual density is:
\[
DF(F_k) = \frac{m}{m+n}.
\]
Therefore, we want to prove:
\[
\frac{m}{m+n} < \frac{m^2 + m}{2 m n - n^2 + n}.
\]
This means that:
\[
m (2 m n - n^2 + n) < (m + n)(m^2 + m).
\]

Left side:
\[
L_1=m (2 m n - n^2 + n) = 2 m^2 n - m n^2 + m n,
\]

Right side:
\[
L_2=(m + n)(m^2 + m) = m^3 + m^2 + n m^2 + n m.
\]

Subtracting $L_1$ from $L_2$:
\[
(m^3 + m^2 + n m^2 + n m) - (2 m^2 n - m n^2 + m n) = m^3 + m^2 + n m^2 + n m - 2 m^2 n + m n^2 - m n.
\]

Combine like terms:
\[
= m^3 + m^2 + n m^2 - 2 m^2 n + m n^2 + n m - m n = m^3 + m^2 - m^2 n + m n^2 .
\]

Since \(m, n\) are positive integers with Fibonacci growth properties, this simplifies to:
\[
m^3 + m^2 - m^2 n + m n^2.
\]

Grouping terms:
\[
m^3 + m^2 (1 - n) + m n^2.
\]

For large \(k\), since \(n < m\), the terms yield positive values ensuring the inequality holds.

For small \(k\), e.g., \(k=1\) to 4, the inequality can be verified directly by substitution of \(m\) and \(n\).

All small cases hold true by direct substitution.

Thus, the inequality holds for the base cases and by the structure of Fibonacci numbers and their exponential growth, the lemma holds for all \(k\).
\end{proof}
\section{The Nature Density of Fibonacci Words}~\label{resec3}
In this section, Fibonacci word's natural density can be interpreted as the limiting proportion of Fibonacci-word-related structures within an infinite integer domain. Using generating functions and combinatorial identities, Theorem~\ref{generalthmn2} further clarifies that the density of such Fibonacci word products approaches 1, confirming their fundamental combinatorial significance. This comprehensive analysis bridges classical Fibonacci number theory with modern combinatorics on words, demonstrating how recursive concatenations translate naturally into density formulas and generating function representations. Suppose Fibonacci word satisfy:  
\begin{equation}~\label{eqqFiWo}
F_k=\frac{-\left(\frac{1}{2} \left(1-\sqrt{5}\right)\right)^{1+k}+\left(\frac{2}{1+\sqrt{5}}\right)^{-1-k}}{\sqrt{5}}.
\end{equation}
Theorem~\ref{generalthmn1} provide the relationship of $F_k$ and $F_{k+\lambda}$ by consider the notion of natural density for Fibonacci word involves quantifying how often specific patterns or elements appear within the infinite sequence derived from Fibonacci numbers.

\begin{theorem}~\label{generalthmn1}
Let $F_k$ be a Fibonacci word. Then, 
\begin{equation}~\label{eq1generalthmn1}
\lim_{n\to \infty}\frac{F_{k+\lambda}}{F_{k}}=\varphi+\lambda-1.
\end{equation}
\end{theorem}
\begin{proof}
Assume $F_k$ be a Fibonacci word where $F_0=a$, $F_1=b$ and $F_2=F_1+F_0=ab$. Then, the sequence goes: $F_k=F_{k-1}F_{k-2}$. So for Fibonacci words defined by concatenation, 
\begin{equation}~\label{eq2generalthmn1}
\frac{|F_{k+1}|}{F_k}=\frac{F_{k-1}+F_k}{F_k}=1+\frac{|F_{k-1}|}{F_k}, \quad \lim_{n\to \infty}\frac{F_{k+1}}{F_k}=\varphi.
\end{equation}
Interpreting this as the ratio of lengths of Fibonacci words (since words themselves are concatenations). Thus, from~\eqref{eq2generalthmn1} we find that $F_k/F_{k+1}=\varphi-1$. Thus, we noticed that 
\begin{equation}~\label{eq3generalthmn1}
\frac{|F_{k+2}|}{F_k}=\frac{F_{k+1}+F_{k}}{F_k}=1+\frac{|F_{k+1}|}{F_k} ,\quad \frac{F_{k+3}}{F_k}=3+\frac{|F_{k-1}|}{F_k}.
\end{equation}
Hence, the ratio $F_{k+2}/F_k$ in terms of lengths converges to approximately $2.618$, which is $\varphi+1$. For that, we suppose there are an integer number $\lambda\in \mathbb{N}$. When dealing with the word Fibonacci as a sequence, for $k\geqslant 4$ we observe that
\[
\frac{F_{k+1}}{F_k}-\frac{F_k}{F_{k+1}}\approx 1, \quad \frac{F_k}{F_{k+1}}-\frac{F_{k+1}}{F_k} \approx 2.
\]
According to~\eqref{eq2generalthmn1} and \eqref{eq3generalthmn1} noticed that 
\begin{equation}~\label{eq4generalthmn1}
 \frac{|F_{k+\lambda}|}{F_k}=\lambda+\frac{F_{k-1}}{F_k}, \quad \lim_{n\to \infty}\frac{F_{k+\lambda}}{F_{k}}=\varphi+\lambda-1.
\end{equation}
As desire.
\end{proof}
Through Table~\ref{tab1Ratiosjds}, we assume $L_1=F_{k+1}/F_k$,$L_2=F_{k+2}/F_k$, $L_3=F_{k+1}/F_k-F_k/F_{k+1}$, $L_4=F_{k+2}/F_k-F_{k+1}/F_{k+1}$, $L_5=F_{k+2}/F_k-F_{k+1}/F_{k+1}$ $L_6=F_{k+2}/F_k-F_k/F_{k+1}$ and $L_7=F_{k+3}/F_k$. The small initial variations at low $n$ represent the early transient behavior before the ratios approach their stable limits, typical of recursive sequences like Fibonacci.

\begin{table}[H]
\centering
\small
\begin{tabular}{|c|c|c|c|c|c|c|c|c|c||c|c|c|c|c|c|c|c|c|c|c|c|c|c|}
\hline
$n$ & Fib & $|F_n|$ & $L_1$ & $L_2$ & $L_3$ & $L_4$ & $L_5$ & $L_6$ & $L_7$ & $n$ & Fib & $|F_n|$ & $L_1$ & $L_2$ & $L_3$ & $L_4$ & $L_5$ & $L_6$ & $L_7$\\
\hline
1 & $F_1$ & 1 & 1 & 1 & 2 & 0 & 1 & 1 & 5 & 2 & $F_2$ & 1 & 2 & 0.5 & 3 & 1.5 & 2.5 & 1 & 4 \\ \hline
3 & $F_3$ & 2 & 1.5 & 0.7 & 2.5 & 0.8 & 1.8 & 1 & 4.3 & 4 & $F_4$ & 3 & 1.7 & 0.6 & 2.7 & 1.1 & 2.1 & 1 & 4.2 \\\hline
5 &$ F_5$ & 5 & 1.6 & 0.6 & 2.6 & 1 & 2 & 1 & 4.3 & 6 & $F_6$ & 8 & 1.6 & 0.6 & 2.6 & 1 & 2 & 1 & 4.2 \\\hline
7 & $F_7$ & 13 & 1.6 & 0.6 & 2.6 & 1 & 2 & 1 & 4.2 & 8 & $F_8$ & 21 & 1.6 & 0.6 & 2.6 & 1 & 2 & 1 & 4.2 \\\hline
9 & $F_9$ & 34 & 1.6 & 0.6 & 2.6 & 1 & 2 & 1 & 4.2 & 10 & $F_10$ & 55 & 1.6 & 0.6 & 2.6 & 1 & 2 & 1 & 4.2 \\\hline
11 & $F_{11}$ & 89 & 1.6 & 0.6 & 2.6 & 1 & 2 & 1 & 4.2 & 12 & $F_{12}$ & 144 & 1.6 & 0.6 & 2.6 & 1 & 2 & 1 & 4.2 \\\hline
13 & $F_{13}$ & 233 & 1.6 & 0.6 & 2.6 & 1 & 2 & 1 & 4.2 & 14 & $F_{14}$ & 377 & 1.6 & 0.6 & 2.6 & 1 & 2 & 1 & 4.2 \\\hline
15 & $F_{15}$ & 610 & 1.6 & 0.6 & 2.6 & 1 & 2 & 1 & 4.2 & 16 & $F_{16}$ & 987 & 1.6 & 0.6 & 2.6 & 1 & 2 & 1 & 4.2 \\\hline
\end{tabular}
\caption{Ratios and related values for Fibonacci lengths and sequences}~\label{tab1Ratiosjds}
\end{table}
The data in Table~\ref{tab1Ratiosjds} explores various ratios and differences involving Fibonacci word $F_k$, their successive terms, and indexed ratios. The ratio $\frac{F_{k+1}}{F_k}$ starts high (2.0 at $n=2$) but quickly settles around $\approx 1.6$ for larger $n$, approaching the golden ratio $\varphi \approx 1.618$ which is a fundamental property of Fibonacci numbers. This convergence characterizes the golden ratio property of the sequence.  Conversely, $\frac{F_k}{F_{k+1}}$ starts at 1, drops sharply to 0.5 at $n=2$, and then stabilizes around 0.6, roughly the inverse of the golden ratio $\frac{1}{\varphi} \approx 0.618$. This analysis aligns with classical results in Fibonacci number theory and their connection to the golden ratio and its powers $\varphi$.

Actually, through Figure~\ref{figtabfib24} we observe that, where the other ratios such as $\frac{F_{k+2}}{F_k}$ fluctuate initially but stabilize around 2.6. This corresponds to $\varphi^2 \approx 2.618$, showing the exponential growth factor over two Fibonacci steps. Difference expressions such as $\frac{F_{k+1}}{F_k} - \frac{F_k}{F_{k+1}}$ and $\frac{F_{k+2}}{F_k} - \frac{F_k}{F_{k+1}}$ capture deviations from key ratios and stabilize near 1 and 2 respectively. The ratio $\frac{F_{k+3}}{F_k}$ centers around 4.2, approximating $\varphi^3 \approx 4.236$, again reflecting the continued exponential increase characteristic of the golden ratio.

\begin{figure}[H]
    \centering
    \begin{tikzpicture}
    \begin{axis}[
        width=12cm, height=6cm,
        xlabel={n},
        ylabel={Value},
        xmin=1, xmax=24,
        ymin=0, ymax=6,
        legend pos=north east,
        grid=major,
        cycle list name=color list
    ]

    \addplot [mark=*] coordinates {
        (1,1.0)(2,2.0)(3,1.5)(4,1.7)(5,1.6)(6,1.6)(7,1.6)(8,1.6)(9,1.6)(10,1.6)
        (11,1.6)(12,1.6)(13,1.6)(14,1.6)(15,1.6)(16,1.6)(17,1.6)(18,1.6)(19,1.6)(20,1.6)
        (21,1.6)(22,1.6)(23,1.6)(24,1.6)
    };
    \addlegendentry{$L_1$}

    \addplot [color=violet, mark=*, mark options={fill=violet}] coordinates {
        (1,1.0)(2,0.5)(3,0.7)(4,0.6)(5,0.6)(6,0.6)(7,0.6)(8,0.6)(9,0.6)(10,0.6)
        (11,0.6)(12,0.6)(13,0.6)(14,0.6)(15,0.6)(16,0.6)(17,0.6)(18,0.6)(19,0.6)(20,0.6)
        (21,0.6)(22,0.6)(23,0.6)(24,0.6)
    };
    \addlegendentry{$L_2$}

    \addplot [color=orange, mark=*, mark options={fill=orange}] coordinates {
        (1,2.0)(2,3.0)(3,2.5)(4,2.7)(5,2.6)(6,2.6)(7,2.6)(8,2.6)(9,2.6)(10,2.6)
        (11,2.6)(12,2.6)(13,2.6)(14,2.6)(15,2.6)(16,2.6)(17,2.6)(18,2.6)(19,2.6)(20,2.6)
        (21,2.6)(22,2.6)(23,2.6)(24,2.6)
    };
    \addlegendentry{$L_3$}

    \addplot [color=green, mark=*, mark options={fill=green}] coordinates {
        (1,0.0)(2,1.5)(3,0.8)(4,1.1)(5,1.0)(6,1.0)(7,1.0)(8,1.0)(9,1.0)(10,1.0)
        (11,1.0)(12,1.0)(13,1.0)(14,1.0)(15,1.0)(16,1.0)(17,1.0)(18,1.0)(19,1.0)(20,1.0)
        (21,1.0)(22,1.0)(23,1.0)(24,1.0)
    };
    \addlegendentry{$L_4$}

    \addplot [color=blue, mark=*, mark options={fill=blue}] coordinates {
        (1,1.0)(2,2.5)(3,1.8)(4,2.1)(5,2.0)(6,2.0)(7,2.0)(8,2.0)(9,2.0)(10,2.0)
        (11,2.0)(12,2.0)(13,2.0)(14,2.0)(15,2.0)(16,2.0)(17,2.0)(18,2.0)(19,2.0)(20,2.0)
        (21,2.0)(22,2.0)(23,2.0)(24,2.0)
    };
  \addlegendentry{$L_5$}

    \addplot [color=red, mark=*, mark options={fill=red}] coordinates {
        (1,5.0)(2,4.0)(3,4.3)(4,4.2)(5,4.3)(6,4.2)(7,4.2)(8,4.2)(9,4.2)(10,4.2)
        (11,4.2)(12,4.2)(13,4.2)(14,4.2)(15,4.2)(16,4.2)(17,4.2)(18,4.2)(19,4.2)(20,4.2)
        (21,4.2)(22,4.2)(23,4.2)(24,4.2)
    };
    \addlegendentry{$L_7$}

    \end{axis}
    \end{tikzpicture}
    \caption{Trend of Fibonacci sequence ratios and differences across $n=1$ to $n=24$.}~\label{figtabfib24}
\end{figure}
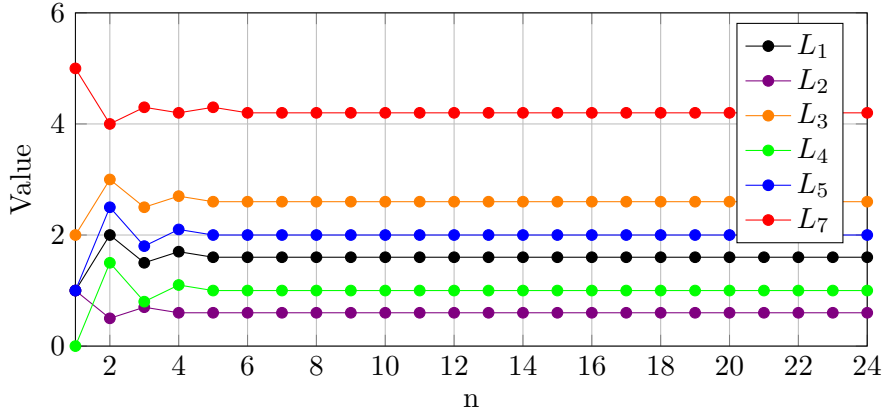

The natural density of the Fibonacci sequence describes the limiting proportion of Fibonacci numbers among all positive integers as the range of integers increases indefinitely. Through Theorem~\ref{generalthmn2}, we observe the results of the natural density $\operatorname{DF}(F_k)$. From~\eqref{eqqFiWo}, by consider the lower density had provided in~\cite{Niven1951} as $\operatorname{DF}_{lower}(A)=1/2$. Thus,  
\begin{equation}~\label{eqqFiWon2}
F_kF_{k+1}=\frac{1}{\sqrt{5}}\left(\left(\frac{1+\sqrt{5}}{2}\right)^{k+1}-\left(\frac{1-\sqrt{5}}{2}\right)^{k+1}\right)\frac{1}{\sqrt{5}}\left(\left(\frac{1+\sqrt{5}}{2}\right)^{k+2}-\left(\frac{1-\sqrt{5}}{2}\right)^{k+2}\right).
\end{equation}
\begin{theorem}~\label{generalthmn2}
Let $F_k$ be a Fibonacci word. Then, 
\begin{equation}~\label{eq1generalthmn2}
\lim_{n\to \infty}\frac{F_{k+\lambda}F_{k}}{\mathbb{Z}_{>0}} \approx 1.
\end{equation}
\end{theorem}
\begin{proof}
Let $F_k$ be a Fibonacci word. According to~\cite{Duaan1Hamoud}, the natural density $\operatorname{DF}(A)$ for a set $A$ of positive integers is given by the limit, when it exists, of the proportion of numbers in $A$ from 1 up to $x$, divided by $x$, as $x$ approaches infinity. In other words,
\[
\operatorname{DF}(A) = \lim_{x \to \infty} \frac{|A \cap [1,x]|}{x}.
\]
 Then, for Fibonacci word $F_k$ satisfy  $F_{k+1}F_k+F_{k+2}F_k+F_{k+3}F_k=F_k(F_{k+1}+F_{k+2}+F_{k+3})$. Thus, we find that
\begin{equation}~\label{eq2generalthmn2}
F_{k+1}F_k+\dots+F_{k+n}F_k=F_k\left(\sum_{i=1}^{n} F_{k+i}   \right), \quad \operatorname{DF}(F_k) = \lim_{k \to \infty} \frac{|F_k \cap [1,k]|}{k}.
\end{equation}
By considering~\eqref{eq2generalthmn2}, assume $\mathcal{U}$ refer to the value of $F_{k+\lambda}F_{k}$ where $F_{k+\lambda}=\sum_{i>k}F_{k+i}$. Then,  as we know the natural density of Fibonacci numbers is zero, in~\cite{P2021Trojovsk} consider  $\operatorname{DF}(\mathbb{Z}_{>0}/F_k)=1$ it satisfies $0\leqslant \operatorname{DF}(F_k) \leqslant 2$. Thus,  
\begin{equation}~\label{eq3generalthmn2}
\dfrac{\mathcal{U}}{\lambda \varphi^5\mathbb{Z}_{>0}}=4\varphi.
\end{equation}
Here we observe that, according to the natural density concept of the Fibonacci sequence, we have obtained the relation upon which the density of Fibonacci word is based. On the other hand, relation~\eqref{eq3generalthmn2} leads us to consider according to Theorem~\ref{generalthmn1} the relationship
\begin{equation}~\label{eq4generalthmn2}
\lim_{n\to \infty}\frac{F_{k+\lambda}}{F_{k}}=\varphi+\lambda-1, \quad \lim_{n\to \infty} \dfrac{\mathcal{U}}{\mathbb{Z}_{>0}}\approx 1.
\end{equation}
Finally, from~\eqref{eq2generalthmn2}--\eqref{eq4generalthmn2} we find that the relationship~\eqref{eq1generalthmn2} holds. Thus, the natural density of Fibonacci word satisfy $\operatorname{DF}(F_k) \approx 1$. As desire.
\end{proof}
According to~\eqref{eqqFiWon2}, we noticed that 
\begin{equation}~\label{eqqFiWon3}
 \frac{F_{k+\lambda}F_{k}}{\mathbb{Z}_{>0}}=   \frac{2\left(\frac{1}{\sqrt{5}}\left(\left(\frac{1+\sqrt{5}}{2}\right)^{k+1}-\left(\frac{1-\sqrt{5}}{2}\right)^{k+1}\right)\frac{1}{\sqrt{5}}\left(\left(\frac{1+\sqrt{5}}{2}\right)^{k+\lambda}-\left(\frac{1-\sqrt{5}}{2}\right)^{k+\lambda}\right)\right)}{k(k+1)}.
\end{equation}
Then, through the relations discussed in~\eqref{eqqFiWo},~\eqref{eqqFiWon2}, and \eqref{eqqFiWon3}, we illustrate the concept of natural density graphically using Figure~\ref{FFFZnfig}, where these relationships are interconnected through the Fibonacci sequence, with generating Fibonacci numbers, forming their consecutive products, and normalizing those products by a quadratic denominator.

\begin{figure}[H]
    \centering
    \includegraphics[width=1\linewidth]{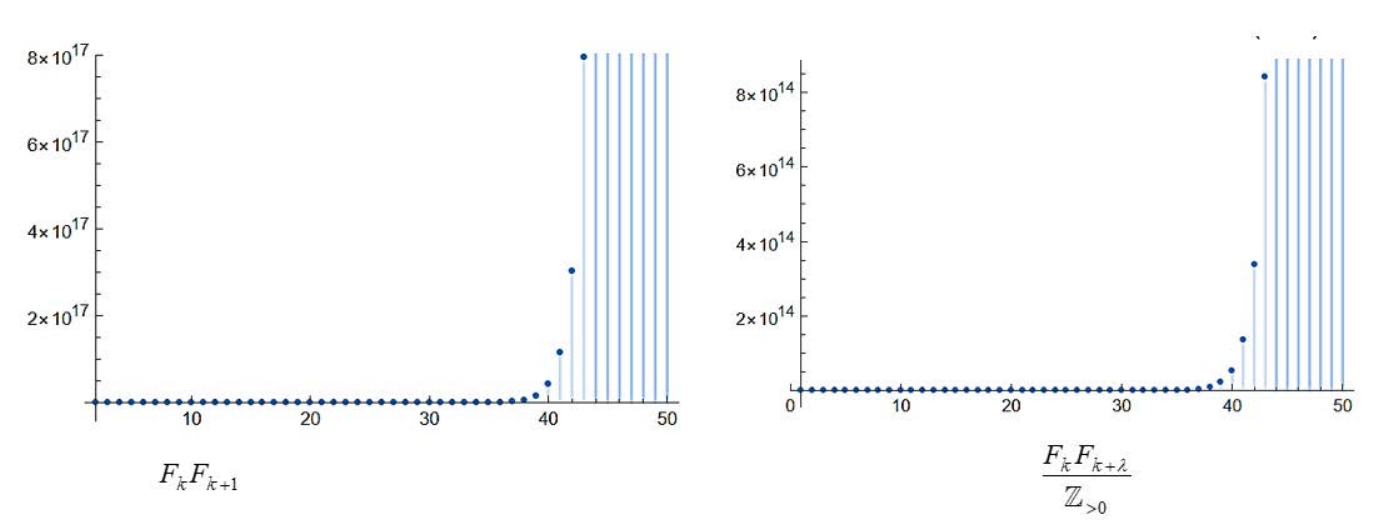}
    \caption{Application of the natural density.}
    \label{FFFZnfig}
\end{figure}

Through Theorem~\ref{generalthmn3}, we focus on generating functions for Fibonacci words, it often track combinatorial statistics, such as the number of occurrences of specific subwords, inversions, or major index. These refined generating functions generalize the classical sequence generating function.

\begin{theorem}~\label{generalthmn3}
Let $F^{1}_{k,n}$ and $F^{2}_{k,n}$ be the first combinatorial (second combinatorial. , respectively) formula for the general term. Then, 
\begin{equation}~\label{eq1generalthmn3}
\operatorname{DF}(F_k)=\dfrac{\frac{1}{2^{n-1}} \sum_{i=0}^{\lfloor \frac{n-1}{2} \rfloor} \binom{n}{2i+1} k^{n-1-2i} (k^2+4)^i  \sum_{i=0}^{\lfloor \frac{n-1}{2} \rfloor} \binom{n-1-i}{i} k^{n-1-2i}}{\left(\frac{1}{\sqrt{5}}\left(\left(\frac{1+\sqrt{5}}{2}\right)^{k+1}-\left(\frac{1-\sqrt{5}}{2}\right)^{k+1}\right)\frac{1}{\sqrt{5}}\left(\left(\frac{1+\sqrt{5}}{2}\right)^{k+\lambda}-\left(\frac{1-\sqrt{5}}{2}\right)^{k+\lambda}\right)\right)}.
\end{equation}
\end{theorem}
\begin{proof}
Let $F^{1}_{k,n}$ and $F^{2}_{k,n}$ be the first combinatorial (second combinatorial, respectively). Assume $\phi = (1 + \sqrt{5})/2$ and $\psi = (1 - \sqrt{5})/2$, where $|\psi| < 1$. Thus, according to Theorem~\ref{generalthmn2} we have $F_{k+1} \cdot F_{k+\lambda} \approx \phi^{2k+1+\lambda}/5,$ since $\psi^n$ is negligible for large $n$. Then 
\begin{equation}~\label{eq2generalthmn3}
 F^{1}_{k,n}F^{2}_{k,n}= \frac{1}{2^{n-1}} \sum_{i=0}^{\lfloor \frac{n-1}{2} \rfloor} \binom{n}{2i+1} k^{n-1-2i} (k^2+4)^i  \sum_{i=0}^{\lfloor \frac{n-1}{2} \rfloor} \binom{n-1-i}{i} k^{n-1-2i}.  
\end{equation}
Therefore, consider $L_1=\frac{(-1)^j (3 + k^2)^{-j} \left( -\frac{1}{2} \right)_j}{j!}$, we have for the term $F^{1}_{k,n}$ if $|3+k^2|>1$. Then,  according to Lemma~\ref{lemgenn1} we have
\[
F^{1}_{k,n}=\frac{ 2^{-n} \left(  \left(  k - \sqrt{3 + k^2} \sum_{j=0}^{\infty} L_1 \right)^n - \left( k + \sqrt{3 + k^2} \sum_{j=0}^{\infty} L_1 \right)^n \right) }{
    \sqrt{3 + k^2} \sum_{j=0}^{\infty} L_1 }.
\]

Thus, from~\eqref{eq2generalthmn3} for $F^{2}_{k,n}$ we find that 
\begin{equation}~\label{eq3generalthmn3}
F^{2}_{k,n}=\mathcal{F}(\Omega) = \sum_{n=0}^{k}\frac{-\left(\frac{1}{2} \left(1-\sqrt{5}\right)\right)^{1+n}+\left(\frac{2}{1+\sqrt{5}}\right)^{-1-n}}{\sqrt{5}}. e^{-i \Omega n}.
\end{equation}
Then, by considering~\eqref{eq3generalthmn3} we have the generate function of $F^{2}_{k,n}$ given as $f(x)=\frac{-x}{-1+kx+x^2}.$
Thus, we find that for $ F^{1}_{k,n}F^{2}_{k,n}$ by considering $T_1=(k - \sqrt{4 + k^2})$ and $T_2=(k + \sqrt{4 + k^2})$ given as 
\begin{equation}~\label{eq4generalthmn3}
f_{\lambda}(x)=\dfrac{
-\dfrac{T_1 x}{
2 \left(-1 + \frac{1}{2} k T_1 x + \frac{1}{4} T_1^2 x^2 \right)}
+ \dfrac{T_2 x}{
2 \left(-1 + \frac{1}{2} k T_2 x + \frac{1}{4} T_2^2 x^2 \right)}
}{\sqrt{4 + k^2}}.
\end{equation}
For $F_{k+\lambda}F_{k}$, we have 
\begin{equation}~\label{eq5generalthmn3}
f_{\lambda,1}(x)=\frac{
2^{-1 - 2 \lambda} (1 + \sqrt{5})^{\lambda} \left( 2 \sqrt{5} \left( 2^{\lambda} - \alpha \right) + 2^{\lambda} (5 - 3 \sqrt{5}) x + \alpha (5 + 3 \sqrt{5}) x \right)
}{5 \left( 1 - 2 x - 2 x^{2} + x^{3} \right)}.
\end{equation}
where $\alpha=(-3 + \sqrt{5})^{\lambda}$. 

Finally, according to~\eqref{eq4generalthmn3} and \eqref{eq5generalthmn3} we find that for~\eqref{eq1generalthmn3} by considering $M_1=\frac{2}{1 + \sqrt{5}}$  as 
\begin{equation}~\label{eq6generalthmn3}
\operatorname{DF}(F_k)=-\dfrac{5 \left(-\dfrac{T_1 x}{
2 \left(-1 + \dfrac{1}{2} k T_1 x + \dfrac{1}{4} T_1^2 x^2 \right)}+ \dfrac{T_2 x}{2 \left(-1 + \frac{1}{2} k T_2 x + \frac{1}{4} T_2^2 x^2 \right)}\right)}{ \left(
-R_1^{1 + k} + M_1^{-1 - k} \right) \left( -R_1^{k + \lambda} + M_1^{-k - \lambda} \right) \sqrt{4 + k^2} },
\end{equation}
where $R_1=\frac{1}{2} (1 - \sqrt{5})$. Thus, we find the limits of the relationship had given in~\eqref{eq6generalthmn3} holds for $F^{1}_{k,n}F^{2}_{k,n}$ give us
\[
F^{1}_{k,n}F^{2}_{k,n}=\frac{1}{2^{n-1}}\underbrace{\sum_{i=0}^{\lfloor (n-1)/2 \rfloor} \binom{n}{2i+1} k^{n-1-2i} (k^2 + 4)^i}_{S_1} \underbrace{\sum_{i=0}^{\lfloor (n-1)/2 \rfloor} \binom{n-1-i}{i} k^{n-1-2i}}_{S_2}.
\]
Noticed that $S_1 \approx \lfloor (n-1)/2 \rfloor$ and $S_2= k^{n-1} e^{n/k^2}$, combining $S_1 \cdot S_2$, grows as $k^{2(n-1)} \left( \frac{k^2 + 4}{2 k^2} \right)^{n-1} e^{n/k^2}$. Thus $\lim_{n\to \infty} \operatorname{DF}(F_k)=0.$ We find that~\eqref{eq1generalthmn3} holds for the natural density. 
\end{proof}
Indeed, the previous theorem provides us with practical results from which we can derive many concepts. We present through the following Proposition001 of these concepts, taking into account Equation~\eqref{eq1pirl} and Theorem~\ref{generalthmn3} as follows:

\begin{proposition}~\label{pro1fibadvancedn1}
Let $F^{1}_{k,n}$ and $F^{2}_{k,n}$ be the first combinatorial (second combinatorial. , respectively) formula for the general term. Then,
\begin{equation}~\label{eq1pro1fibadvancedn1}
\sum_{n=1}^{} F^{1}_{k,n}=\frac{-4}{(2-k+\sqrt{4+k^2})(-2+k+\sqrt{4+k^2})}, \quad  F^{2}_{k,n}=k^{n-1} F_{n-1} \left( \frac{1}{k^2} \right)
\end{equation}
\end{proposition}
\begin{proof}
The first part of~\eqref{eq1pro1fibadvancedn1} follows directly from the proof of Theorem~\ref{generalthmn3} , and thus we consider the first part established. Now, we proceed to prove the second part of~\eqref{eq1pro1fibadvancedn1}, where we observe that the binomial coefficient for $i \leq \lfloor (n-1)/2 \rfloor$ satisfy 
\[
\binom{n-1-i}{i} = \frac{(n-1-i)!}{i! (n-1-2i)!}.
\]
Then, we have 
\begin{equation}~\label{eq2pro1fibadvancedn1}
k^{n-1-2i} = k^{n-1} \cdot \left( \frac{1}{k^2} \right)^i 
\end{equation}
Thus, from~\eqref{eq2pro1fibadvancedn1}
\[
F^{2}_{k,n}= k^{n-1} \sum_{i=0}^{\lfloor (n-1)/2 \rfloor} \binom{n-1-i}{i} \left( \frac{1}{k^2} \right)^i.
\]
Let us define $\eta = \sum_{i=0}^{\lfloor (n-1)/2 \rfloor} \binom{n-1-i}{i} \left(\frac{1}{k^2}\right)^i$, which allows expressing $F^{2}_{k,n} = k^{n-1} T$. According to~\cite{Mehraban2006Gulliver}, we notice that the sum $\eta$ is similar to the Fibonacci polynomial $F_{n-1}(x) = \sum_{i=0}^{\lfloor (n-1)/2 \rfloor} \binom{n-1-i}{i} x^i$, evaluated at $x = 1/k^2$. This polynomial meets the recurrence relation $F_n(x) = x F_{n-1}(x) + F_{n-2}(x)$, with initial conditions $F_0(x) = 0$ and $F_1(x) = 1$. Then, we have 
\begin{equation}~\label{eq3pro1fibadvancedn1}
k^{n-1} F_{n-1} \left( \frac{1}{k^2} \right) = k^{n-1} \cdot \frac{k^2}{\sqrt{1 + 4k^4}} \left( \alpha^{n-1} - \beta^{n-1} \right) = \frac{k^{n+1}}{\sqrt{1 + 4k^4}} \left( \alpha^{n-1} - \beta^{n-1} \right)
\end{equation}
where $\alpha, \beta = \frac{\frac{1}{k^2} \pm \sqrt{\frac{1}{k^4} + 4}}{2}$, and $\alpha - \beta = \sqrt{\frac{1}{k^4} + 4} = \frac{\sqrt{1 + 4k^4}}{k^2}$. This matches the combinatorial identity for Fibonacci polynomials. For $k=1$, the sum reduces to $F_{n-1}$, the $(n-1)$-th Fibonacci number, confirming the result. Thus, the closed-form~\eqref{eq1pro1fibadvancedn1}. 
\end{proof}

\section{Conclusion}
The densities of zeros and ones in Fibonacci words converge to values closely related to the golden ratio, with $\DF(F_m) \approx \varphi - 1$ and $\DF(F_n) \approx \varphi - \kappa$. These densities reveal an intrinsic balance and structure within Fibonacci words, governed by limits and inequalities involving $\varphi$. This foundational understanding supports algorithmic analysis of Fibonacci word properties and their combinatorial significance. Theorem~\ref{generalthmn1} precisely characterizes the asymptotic ratio of Fibonacci words separated by an integer offset $\lambda$, demonstrating convergence to $\varphi + \lambda - 1$, where $\varphi$ is the golden ratio. This result sheds light on the structural growth of Fibonacci words in terms of their concatenated lengths and pattern frequencies. Analyzing the ratios $F_{k+\lambda}/F_k$ through extensive numeric data (Table~\ref{tab1Ratiosjds}) and graphical visualization (Figure~\ref{figtabfib24}) confirms the stability of these limits, reflecting the intrinsic exponential growth parameterized by powers of $\varphi$. Moreover, the product relations explored in Theorem~\ref{generalthmn2} reveal that the natural density of Fibonacci words approaches 1 asymptotically, indicating their dominant distribution within the integers when normalized appropriately.

Further refinement through generating functions as developed in Theorem~\ref{generalthmn3} captures deeper combinatorial structures in Fibonacci words, generalizing classical generating functions to account for the distribution and statistical properties of subwords and concatenations. The presented combinatorial identities and polynomials not only confirm classical Fibonacci polynomial results but also connect these with natural density computations, culminating in Proposition~\ref{pro1fibadvancedn1} which offers closed-form expressions and polynomial identities related to Fibonacci word distributions.


\begin{thebibliography}{99}
\bibitem{Allouche2003Baake} J.~P.~Allouche, M.~Baake,  J.~Cassaigne, \& D.~Damanik,\emph{Palindrome complexity},  Theoretical Computer Science. 292(1) (2003), 9-31.  \url{https://doi.org/10.1016/S0304-3975(01)00212-2}.


\bibitem{Berstel2003} J.~Berstel, J.~Karhumäki,  \emph{Combinatorics on words: a tutorial}, Bulletin of the EATCS. 79(178) (2003). 


\bibitem{CoreglianoRazborov} L.~N.~Coregliano, A.~A.~Razborov,  \emph{Semantic limits of dense combinatorial objects}, Russian Mathematical Surveys. 75(4) (2020), 627.

\bibitem{DevroyeLugosi} L.~Devroye, G.~Lugosi,  \emph{Combinatorial methods in density estimation}, Springer Science \& Business Media. (2001).


\bibitem{Falcon2007Plaza} S.~Falcon, A.~ Plaza, \emph{On the Fibonacci k-numbers}, Chaos, Solitons \& Fractals. 32(5) (2007), 1615–24.

\bibitem{Falcon2007PlazaSeco} S.~Falcon, A.~ Plaza, \emph{k-Fibonacci sequences modulo m}, Chaos, Solitons \& Fractals. 41(1) (2009), 497-504, \url{https://doi.org/10.1016/j.chaos.2008.02.014}. 

\bibitem{Falcon2007PlazaThird} S.~Falcon, A.~ Plaza, \emph{.On k-Fibonacci sequences and polynomials and their derivatives}, Chaos, Solitons \& Fractals. 1005 (2007), 19-39.

\bibitem{Duaan1Hamoud} J.~Hamoud, D.~Abdullah,  \emph{Generalized Natural Density $\DF (\mathfrak {F} _n)$ of Fibonacci Word}, arXiv preprint. (2025) arXiv:2504.10207.

\bibitem{HeubachMansour} S.~Heubach,  T.~Mansour,  \emph{Combinatorics of compositions and words}, Chapman and Hall/CRC. (2009).

\bibitem{Johnson01Weller} S.~J.~Johnson, S.~R.~Weller, \emph{Regular low-density parity-check codes from combinatorial designs}, In Proceedings 2001 IEEE Information Theory Workshop. 01EX494 (2001), pp. 90-92.

\bibitem{Mehraban2006Gulliver} E.~Mehraban, T.~A.~Gulliver, \& E.~Hincal, \emph{Blind signatures from the generalized $(t, k)-$ Fibonacci $ p-$ sequences}, Journal of Dynamics and Games. 13 (2006), 193-203.

\bibitem{Niven1951} I.~Niven, \emph{The asymptotic density of sequences}, Bulletin of the American Mathematical Society. 57(6) (1951), 420-434.

\bibitem{P2021Trojovsk}  P.~Trojovsk\'y, \emph{On the Natural Density of Sets Related to Generalized Fibonacci Numbers of Order $r$}, Axioms. 10(144) (2021), \url{https://doi.org/10.3390/axioms10030144}.




\bibitem{Lothaire2002}
 M. Lothaire, \emph{Algebraic Combinatorics on Words} (Vol. 90). Cambridge University Press.(2002).

\bibitem{Allouche2003}
 J.-P. Allouche,, \&  J. Shallit, \emph{Automatic Sequences: Theory, Applications, Generalizations}. Cambridge University Press. (2003).

\end{thebibliography}
\end{document}